\documentclass[12pt]{article}
\usepackage[utf8]{inputenc}
\usepackage{amsmath, amsthm,  amsfonts, amssymb}
\usepackage{commath, mathabx}
\usepackage{MnSymbol}
\usepackage[a4paper, margin=1in]{geometry}
\usepackage{graphicx}
\usepackage{ upgreek }
\graphicspath{ {./images/} }
\usepackage{float}
\usepackage{listings}
\usepackage{color}
\usepackage{amsthm}
\usepackage{enumitem}
\usepackage{mathrsfs}
\usepackage{hyperref}
\usepackage{url}

\makeatletter
\newcommand*{\rom}[1]{\textnormal{(\expandafter\romannumeral #1)}}
\makeatother

\newcommand{\enum}{\textnormal{\roman*)}}

\theoremstyle{plain}
\newtheorem{thm}{Theorem}
\newtheorem{prop}[thm]{Proposition}
\newtheorem{lem}[thm]{Lemma}
\newtheorem{cor}[thm]{Corollary}

\newtheorem*{thm*}{Theorem}

\theoremstyle{definition}
\newtheorem{defn}[thm]{Definition}

\renewenvironment{proof}{{\noindent \bfseries Proof }}{\qed}

\usepackage{tikz}
\usetikzlibrary{decorations.markings}
\usetikzlibrary{calc}
\usetikzlibrary{automata,positioning,arrows}
\usetikzlibrary{shapes}
\usetikzlibrary{tqft}
\usetikzlibrary{matrix}

\title{\textbf{Totally Disconnected (non-metric) Gelfand Duality}}
\author{
  Sebastián Rodríguez \\
  Universidad de los Andes \\
  \texttt{js.rodriguezl@uniandes.edu.co} 
  \and
  Xavier Caicedo \\
  Universidad de los Andes \\
  \texttt{xcaicedo@uniandes.edu.co}
}
\date{}

\begin{document}

\maketitle

 \noindent \textbf{Abstract: } We characterize those algebras over a disconnected uniformly complete topological field which are representable as algebras of continuous functions on compact topological spaces, generalizing thus Gelfand duality for non-archimedean normed fields (Van der Put's theorem \cite{VDP}). More generally, we establish for any topological field $F$ a (dual) adjunction between the category of compact $F$-Tychonoff spaces and a natural category of commutative $F$-algebras, which becomes a duality for normed fields satisfying a Stone-Weierstrass theorem. To obtain these results we do not utilize analytic tools, but the canonical group uniformity of the field and intrinsic properties of the algebras.

\section*{Introduction}

Gelfand duality is one the finest pieces of 20th century mathematics. Having as cornerstones the Gelfand-Kolmogorov and the Gelfand-Naimark theorems, which establish a correspondence between compact Hausdorff spaces and commutative $C^*$-algebras through the functors sending a space $X$ to the algebra of continuous functions $C(X,\mathbb{C})$, and a commutative $C^*$-algebra $A$ to its maximal spectrum $\textnormal{Max}(A)$, respectively.

This correspondence is not an exclusive phenomenon of the complex numbers; it holds for any topological field. As shown by Vechtomov (\cite{Vechtomov}, Proposition 8.2) the category of compact $F$-Tychonoff spaces—that is, spaces in which the continuous $F$-valued functions separate points and closed sets—is dual to the category of algebras of $F$-valued continuous functions over these spaces (cf. Theorem~\ref{thm : GD}). However, the characterization of the later algebras is known only in the real case (Banach algebras where $a^2+1$ is always invertible) or for non-archimedean valued fields (algebras with many idempotents).



Within this general framework, a natural question arises: \emph{how can we characterize the categories involved for a given topological field $F$?} On the topological side, compact $F$-Tychonoff are, in most cases, fairly simple to characterize (Theorem~\ref{thm:KH}), with their structure largely depending on the connectedness of $F$.



While the topological side of the duality is largely well understood, the category of algebras has been characterized only in the case of complete valued fields.  

This paper aims to provide a purely algebraic characterization of the $F$-algebras of continuous functions, without assuming that the field is normed. We find such a characterization when the field is complete for its canonical group uniformity and is either totally disconnected (cf. Theorem~\ref{thm:GND}) or satisfies the Stone-Weierstrass theorem (cf. Corollary~\ref{cor : ContAlgWS}). The first case extends Van der Put theorem \cite{VDP} by relying solely on the algebraic structure and the topology of the field, thus avoiding norms and deep analytical techniques, in contrast with the approaches in \cite{VDP} and \cite{Dominguez}.

To give a complete description of the algebras involved, we identify the key properties that characterize an $F$-algebra as an algebra of continuous functions on its maximal spectrum. An $F$-algebra $A$ is called a \emph{Gelfand algebra} if, for every maximal ideal $M\subset A$, the quotient $A/M = F$. This turns out that this is the single most important property. This property allows us to endow the maximal spectrum with a topology finer than the Zariski topology, which, specializes to the Gelfand topology used in the proof of the Gelfand-Naimark theorem \cite{Rudin}. Furthermore, we introduce a canonical uniform structure on the algebra, demonstrating that its analytic aspects emerge naturally from its algebraic properties.  



After establishing the key results on algebras of continuous functions and Gelfand algebras, we construct a (dual) adjunction between the category of compact $F$-Tychonoff spaces and a category of $F$-algebras that retain the essential structural features of algebras of continuous functions, thereby extending Gelfand duality. When the field $F$ satisfies the Stone–Weierstrass theorem, except in the case of the complex numbers, the adjunction becomes a duality.

Our approach relies only on canonical constructions derived from the algebra and the topology of the field, highlighting its intrinsic algebraic nature. The algebricity of Gelfand duality was first observed by Mulvey \cite{Mulvey}, but his generalization, as well as that of Vechtomov, omits certain structural properties of algebras of continuous functions that we explicitly preserve.

 In Section~\ref{sec : 1}, we establish the main facts about algebras of continuous functions and give explicit constructions of the functors and natural transformations involved in the duality. 

In Section~\ref{sec : 2} we define the Gelfand topology on $\textnormal{Max}(A)$ and show that, for algebras of the form $C(X,F)$ where $X$ is a compact $F$-Tychonoff space, it coincides with the Zariski topology. Additionally, we introduce the canonical uniformity over a Gelfand algebra, and show that, in the case of $C(X,F)$, the uniform topology is the compact-open topology. We show that the algebraic and the uniform structures interact nicely, generalizing well-known results about $C^*$-algebras and $\mathbb{C}$-algebras to the setting of Gelfand algebras.

 In Section~\ref{sec : 4}, we generalize Van der Put theorem to arbitrary disconnected topological fields. Considering disconnected fields complete with respect to their canonical group uniformity and, using the tools developed in Section~\ref{sec : 2}, we characterize the category of algebras of continuous $F$-valued functions over Stone spaces. 

 In Section~\ref{sec : 3}, we provide a detailed proof of Gelfand duality. Furthermore, we establish a (dual) adjunction between the category of compact $F$-Tychonoff spaces and a category of algebras sharing some of the fundamental properties with algebras of continuous functions. This adjunction extends Gelfand duality and it becomes a duality when the Stone-Weierstrass theorem holds. 

\section{Preliminaries}

\label{sec : 1}

\noindent The aim of the this section is to establish the main ingredients involved in the construction of the duality, namely, the continuous functions functor $\textsf{C}_F$, the maximal spectrum functor $\textsf{M}_F$ and the Gelfand transform $G_F$, except for the Gelfand map $I_F$ which will be defined in Section~\ref{sec : 2}. 

By a topological field, we always mean a proper topological field; that is, one whose topology is not the indiscrete topology. For the reader interested in a more in-depth study of topological fields, we recommend consulting \cite{MR1075419, Warner,Witold}.

\subsection{Generalized Gelfand-Kolmogorov transform}

Let $X$ be a topological space and $F$ a (proper) topological field. We denote by $C(X,F)$ the $F$-algebra of continuous $F$-valued functions with the operations calculated pointwise. 

The maximal spectrum of a commutative unitary ring $R$ is the set $\textnormal{Max}(R) = \{M\subset R: M \textnormal{ maximal ideal}\}$ endowed with the Zariski topology, $\uptau_Z$. Its closed sets have the form $V_I = \{M\in\textnormal{Max}(R): I\subset M\}$, where $I\subset R$ is an ideal, and the canonical open basis is given by the sets $D(a) = \{M\in \textnormal{Max}(R): a\not\in M\}$ with $a\in R$. Notice that for any $a,b\in R$ $D(ab) = D(a)\cap D(b)$ and $D(1)=\textnormal{Max}(A)$ and $D(0) = \emptyset$. The space $\textnormal{Max}(A)$ is always compact and $T_1$ and a subspace of the prime spectrum $\textnormal{Spec}(R)$, the set of prime ideals of $R$ endowed with a similar topology. 

For any topological space $X$ and any topological field $F$ define the \emph{Gelfand transform} as the map

\[
\begin{matrix}
G_F: & X & \to & \textnormal{Max}(C(X,F)) \\
     & x & \mapsto & M_x = \textnormal{ker}(ev_x) 
\end{matrix}
\]
where $ev_x(f) = f(x)$.

 The Gelfand transform is always continuous (\cite{Xavier}, Lemma 2.1) and it is surjective if $X$ is compact Hausdorff (\cite{Xavier}, Theorem 3.2). Furthermore, it is easy to verify that $G_F$ is injective iff $C(X,F)$ separates
points. But in order to make $G_F$  a homeomorphism we require the space $X$ to satisfy a stronger separation axiom. 

\begin{defn}
    Let $F$ be a topological field. A space $X$ is \emph{$F$-Tychonoff} if $X$ is $T_1$  and for all $C\subset X$ closed and $x\in X\setminus C$ there exists $f\in C(X,F)$ such that $f(x)=1$ and $f(C)=0$.
\end{defn}

\noindent Equivalently, a space $X$ is $F$-Tychonoff if for all $x\in X$ and $U$ open with $x\in U$ there exists $f\in C(X,F)$ such that $f(x) = 1$ and $\textnormal{supp}(f)\subset U$, where $\textnormal{supp}(f)$ denotes the closure of the support of the function. By definition $C(X,F)$ separates points so in this case, and since any (proper) topological field is  Hausdorff \cite{Witold}, we have that any $F$-Tychonoff space is Hausdorff.

\begin{lem}
    Let $X$ be $F$-Tychonoff, then $G_F(X)$ is Hausdorff in the inherited topology from \textnormal{Max}$(C(X,F))$.
\end{lem}

\begin{proof}
    Given two points $x,y\in X$ it is enough to find two functions $f,g\in C(X,F)$ such that $f(x),g(y) \neq 0$ and $fg=0$, equivalently, $M_x\in D(f)$, $M_y\in D(g)$ and $D(f)\cap D(g) = D(fg) = \emptyset$. Let $U,V\subset X$ be disjoint open sets such that $x\in U$ and $y\in V$, since $X$ is $F$-Tychonoff we can find $f,g\in C(X,F)$ such that $f(x)=g(y)=1$, $\textnormal{supp}(f)\subset U$ and $\textnormal{supp}(g)\subset V$. These are the desired functions. 
\end{proof}

\vspace{2mm}

If the $F$-Tychonoff space $X$ is compact then $G_F$ is a continuous bijection, since $C(X,F)$ separate points. Hence, by the previous Lemma it is a homeomorphism, which gives us one of the directions of the following theorem.  

\begin{thm}[Vechtomov, \cite{Vechtomov}]
    The Gelfand transform $G_F$ is a homeomorphism iff $X$ is compact and $F$-Tychonoff.
    \label{thm: TGK}
\end{thm}

\begin{proof}
    Suppose that $G_F$ is a homeomorphism, then clearly $X$ is compact. Let $C\subset X$ closed and $y\in X\setminus C$, identifying through $G_F$ we have that $A=\{M_x: x\in C\}$ is a closed set that does not contain $M_y$. Find an ideal $J\subset C(X,F)$ such that $V_J=A$, since $M_y\not \in V_J$ then $J\not\subset M_y$ or, equivalently, there exists $f\in J$ such that $f(y)\neq 0$. But for all $x\in C$ we have $f\in M_x$ ($J\subset M_x$), which implies $f(x)=0$ for all $x\in C$. Summarizing, we have a continuous function $f$ that vanishes on $C$ but not at $y$.
\end{proof}


\vspace{2mm}

Let $\textsf{KH}$ be the category of compact Hausdorff spaces and continuous maps, and given a topological field $F$ let $\textsf{KH}_F$ denote the full subcategory of compact $F$-Tychonoff spaces. \textsf{Stone} will denote full subcategory of compact $0$-dimensional spaces. 


\begin{lem}
Let $F$ be a topological field then, $\textsf{KH}_F$ is determined by the connectivity of the field:
\begin{enumerate}[label=\enum]
    \item  $\textsf{KH}_F=\textsf{KH}$ iff $F$ is path-connected. 
    \item  $\textsf{KH}_F = \textsf{Stone}$ if $F$ is disconnected.
    \item  $\textsf{Stone}\subset\textsf{KH}_F \subset$ totally path-disconnected spaces in \textsf{KH}, if $F$ is connected but path-disconnected.
\end{enumerate}
\label{thm:KH}
\end{lem}

\begin{proof} 
First, note that if a field is not connected, then it is totally disconnected (see \cite{Witold}, Ch. 10, Th. 1). With a slight modification of this proof, one can show that if a field is not path-connected, then it is totally path-disconnected.
\begin{enumerate}[label=\enum]
    \item This follows from Theorem~\ref{thm: TGK} and (\cite{Xavier}, Theorem 2.4).
    \item Let $X$ be a $F$-Tychonoff then $X$ must be totally disconnected, otherwise, if $ C\subset X$ is a non-trivial connected subset and $x\neq y\in C$ then, by the observation at the beginning, for all $f\in C(X,F)$ we have $f(x)=f(y)$, contradicting the fact that $C(X,F)$ separates points. Reciprocally, let $X$ be a Stone space, $C\subset X$ closed and $x\in X\setminus C$, then $x\in U\subset X\setminus C$ with $U$ clopen, define $f: X\to F$ by $f(a)=1$ if $a\in U$ and $f(a)=0$ if $a\in X\setminus U$. Then $f\in C(X,F)$  separates $x$ and $C$. 
    \item  The first inclusion is shown in the second part of \rom{2}, the second one as in the first part of \rom{2}, utilizing that path-disconnected fields are totally path-disconnected. 
\end{enumerate}
\end{proof}

\vspace{2mm}

\noindent  For examples of connected and totally path-disconnected fields see  \cite{Dieudonne} or \cite{Shell}. We wonder if in the last case of Lemma~\ref{thm:KH} the category $\textsf{KH}_F$ is independent of the field as in cases \rom{1} and \rom{2}.

\subsection{\texorpdfstring{The functor \textsf{C}$_F$}{The functor CF}}

For each topological field $F$ the homeomorphism given by the Gelfand transform is part of an equivalence of categories. Denote by $\textsf{Alg}_F$ the category of unitary commutative $F$-algebras and $F$-algebra morphisms preserving the identity. Then the functor $\textsf{C}_F:\textsf{KH}_F\to\textsf{Alg}_F$ given by

\[
\begin{matrix}
    X & \mapsto & C(X,F) & & & (f:X\to Y) & \mapsto & 
    \begin{pmatrix}
        f^*: C(Y,F) & \to & C(X,F) \\
        \varphi & \mapsto & \varphi \circ f
    \end{pmatrix}
\end{matrix}
\]
is a full contravariant embedding from $\textsf{KH}_F$ to $\textsf{Alg}_F$.

Before proceeding with the proof, we must first establish some facts about algebras of continuous functions. Given a topological space $X$ and a topological field $F$, the algebra of continuous functions $C(X,F)$ is clearly commutative and unitary. However, there are additional properties that play a significant role in Gelfand duality and must be satisfied by $C(X,F)$ when $X$ is compact Hausdorff. 

Recall that a $F$-algebra $A$ is \emph{semisimple} if $\textnormal{Jrad}(A) = \bigcap_{M\in\textnormal{Max}(A)} M = \{0\}$. For $A$ unitary the field $F$ embeds canonically into $A$ by $\lambda \mapsto \lambda\cdot 1$. The \emph{spectrum} of an element $a\in A$ is the set $\sigma(a) = \{\lambda\in F : a-\lambda \textnormal{ is not invertible}\}$. It is easy to verify that $\lambda\in \sigma(a)$ iff there exists $M\in\textnormal{Max}(A)$ such that $a + M = \lambda + M$. Finally, a $F$-algebra $A$ is \emph{Gelfand} if $A/M = F$ for every $M\in\textnormal{Max}(A)$. 

By the Gelfand-Mazur theorem \cite{Witold} any Banach $\mathbb{C}$-algebra is a Gelfand algebra, and a Banach $\mathbb{R}$-algebra $A$ is Gelfand iff $a^2+1$ is invertible for all $a\in A$\footnote{If we remove the norm, the condition that $a^2+1$ is invertible for every element of the algebra is equivalent to the quotient by any maximal ideal being a purely transcendental extension of $\mathbb{R}$.}.

\begin{lem}
Let $X$ be a compact Hausdorff space and $F$ be a topological field. Then $C(X,F)$ satisfies the following properties:
\begin{enumerate}[label=\enum]
    \item is Gelfand.
    \item is semisimple.
    \item the spectrum of every element is compact and non empty.
\end{enumerate}

\noindent Furthermore, if $X$ is $F$-Tychonoff then

\begin{enumerate}[label = \enum]\addtocounter{enumi}{3}
    \item every prime ideal is contained in a unique maximal ideal.
\end{enumerate}

\label{lem : propCAlg}

\end{lem}

\begin{proof}
   By (\cite{Xavier}, Theorem 3.2) the Gelfand transform is surjective. Then every maximal ideal $M$ is of the form $M=M_x=\{f\in C(X,F): f(x)=0\}$ for some $x\in X$. \rom{1} follows from the fact that $f-f(x)\in M_x$. Notice that $f\in M_x$ for all $x\in X$ iff $f(x)=0$ for all $x\in X$, from which \rom{2} follows. Finally, for $f\in C(X,F)$ we have $\frac{1}{f}$ exist iff $f(x)\neq 0$ for all $x\in X$. Then $\lambda\in \sigma(f)$ iff there exists $x\in X$ such that $f(x)=\lambda$, i.e. $\sigma(f)=f(X)$.

   If $X$ is $F$-Tychonoff then by Theorem~\ref{thm: TGK} $\textnormal{Max}(C(X,F))$ is Hausdorff. Suppose there exists $P\subset M,N$  with $P$ prime and $N,M$ maximal ideals, let $f,g \in C(X,F)$ such that $M\in D(f)$, $N\in D(g)$ and $D(fg) = \emptyset$. Then $M_x\nin D(fg)$ for all $x\in X$, so $fg = 0$. This implies $fg\in P$ and, since $P$ is prime, we have $f\in P$ or $g\in P$ which leads to a contradiction since $f\nin M$ and $g\nin N$. 
\end{proof}

\begin{prop}
    The functor $\textsf{C}_F$ is one to one in objects (modulo isomorphism) and, moreover, is full and faithful. 
    \label{cor: Cfull}
\end{prop}
 
\begin{proof}
Faithfulness is immediate because $f(x) \neq g(x)$ for $f,g:X\to Y$ implies there is a $\varphi: Y\to F$ such that $\varphi(f(x)) \neq \varphi(g(x))$. Let $X,Y\in\textsf{KH}_F$ and $\psi: C(X,F)\to C(Y,F)$ be a $F$-algebra morphism. It follows from the Gelfand property that any maximal ideal $M\subset C(Y,F)$ induces a morphism $\pi_M: C(Y,F)\to F$, it follows immediately that $\psi^{-1}(M) = \textnormal{ker}(\pi_M\circ\psi)$, thus the preimage of a maximal ideal is a maximal ideal. The map $\psi^{-1}:\textnormal{Max}(C(Y,F))\to \textnormal{Max}(C(X,F))$ is continuous since $\psi^{-1}:\textnormal{Spec}(C(Y,F))\to \textnormal{Spec}(C(X,F))$ is continuous. 

To show fullness, we show that the following diagram commutes, which implies $\psi = \textsf{C}_F(G_X^{-1} \circ \psi^{-1}\circ G_Y)$.

\begin{center}
    \begin{tikzpicture}
        \matrix (m) [matrix of math nodes, name = m, row sep = 1.5cm, column sep = 1cm]
        {
        C(X,F) & C(X,Y) \\
        C(\textnormal{Max}(C(X,F)),F) & C(\textnormal{Max}(C(Y,F)),F) \\
        };
        \path[-stealth]
        (m-1-1) edge node[above]{$\psi$} (m-1-2)
        (m-2-1) edge node[left]{$\textsf{C}_FG_X$} (m-1-1)
        (m-2-2) edge node[right]{$\textsf{C}_FG_Y$} (m-1-2)
        (m-2-1) edge node[below]{$\textsf{C}_F\psi^{-1}$} (m-2-2)
        ;
    \end{tikzpicture}
\end{center}

Take $\varphi = \textsf{C}_F(G_X^{-1} \circ \psi^{-1}\circ G_Y)$ and let $f\in C(X,F)$ and $y\in Y$. Notice that it suffices to show that $\psi(f)(y) = 0$ iff $\varphi(f)(y)=0$, since if $\psi(f)(y) = b$ then $0 = \psi(f)(y) - b = \psi(f-b)(y)$ and by hypothesis this happens iff $0 = \varphi(f-b)(y) = \varphi(f)(y) - b$. If $\psi(f)(y)$ we have $f\in \psi^{-1}(M_y)$, therefore, $ 0 = f(G_X^{-1}(\psi^{-1}(G_Y(y)))) = \varphi(f)(y)$. The other direction is analogous. 
\end{proof}

\vspace{2mm}

 The previous proposition already indicates that the functor $\textsf{C}_F$ "has an inverse" which we describe next. Nevertheless, we postpone this discussion until Section~\ref{sec : 3}.

\vspace{2mm}

\subsection{\texorpdfstring{The functor $\textsf{M}_F$}{The functor MF}}

Recall that in general, taking the maximal spectrum of a ring is not functorial, because given a ring homomorphism $\psi: R\to S$ the inverse image $\psi^{-1}(M)$ of a maximal ideal $M\subset S$ is prime but not necessarily maximal in $R$. Thus we do not have a canonical way to assign a maximal ideal from $R$ to a maximal ideal from $S$, even less to do it continuously. 

In Proposition~\ref{lem : propCAlg} we observed that for an algebra morphism $\psi : C(X,F)\to C(Y,F)$ we have $\psi^{-1}: \textnormal{Max}(C(Y,F))\to \textnormal{Max}(C(X,F))$ is a continuous function using the fact that maximal ideals can be identified with morphism to the field. Notice that this argument works the same for any Gelfand algebra. With this in mind we define the functor $\textsf{M}_F: \mathcal{G}\textsf{Alg}_F\to \textsf{Top}$ by

\[
\begin{matrix}
    A & \mapsto & \textnormal{Max}(A) & & & (\psi:A\to B) & \mapsto & \left(
    \begin{matrix}
        \textnormal{\textsf{M}}_F\psi: \textnormal{Max}(B) & \to & \textnormal{Max}(A) \\
       \hspace{5mm} M & \mapsto & \psi^{-1}(M)
    \end{matrix}
    \right)
\end{matrix}
\]
where $\mathcal{G}\textsf{Alg}_F$ is the category of Gelfand $F$-algebras and $\mathsf{Top}$ is the category of topological spaces. 

Denote the category generated by the image of $\textsf{C}_F$ by $\textsf{CAlg}_F$. By the Gelfand-Kolmogorov theorem if $A\in \textsf{CAlg}_F$ then $\textnormal{Max}(A)\in \textsf{KH}_F$ and by Lemma~\ref{lem : propCAlg} $\textsf{CAlg}_F\subset \mathcal{G}\textsf{Alg}_F$, consequently, we can restrict $\textsf{M}_F$ to $\textsf{CAlg}_F$ in order to get the functor $\textsf{M}_F : \textsf{CAlg}_F \to \textsf{KH}_F$.

\vspace{1mm}

\noindent \textbf{Remark: } There is an alternative way to define the maximal spectrum functor using a different fundamental property. A ring is a $pm$-ring if every prime ideal is contained in a unique maximal ideal. If $R$ is a $pm$-ring the map $\mu_R : \textnormal{Spec}(R)\to\textnormal{Max}(R)$ that assigns every prime ideal to the unique maximal containing it is a continuous retraction (\cite{OrsattiDeMarco}, Theorem 1.2). 

 Denote by $pm\textsf{Alg}_F$ the category of $F$-algebras that are $pm$-rings. We define the (contravariant) functor $\textsf{M}_F: pm\textsf{Alg}_F\to\textsf{KH}$ by
\[
\begin{matrix}
    A & \mapsto & \textnormal{Max}(A) & & & (\psi:A\to B) & \mapsto & \left(
    \begin{matrix}
        \textnormal{\textsf{M}}_F\psi: \textnormal{Max}(B) & \to & \textnormal{Max}(A) \\
       \hspace{5mm} M & \mapsto & \mu_A(\psi^{-1}(M))
    \end{matrix}
    \right)
\end{matrix}
\]

In fact, let $R$ be a $pm$-ring then by (\cite{DominguezGomez1981}, Prop 1) $(\textnormal{Max}(A),\uptau_Z)$ is Hausdorff. The continuity of $\textsf{M}_F\psi$ then follows from the continuity of $\textnormal{Max}(B)\overset{i}{\hookrightarrow}\textnormal{Spec}(B)\overset{\psi^{-1}}{\rightarrow}\textnormal{Spec}(A)$ and $\textnormal{Spec}(A)\overset{\mu_A}{\rightarrow}\textnormal{Max}(A)$. Given $A \overset{\psi}{\rightarrow} B \overset{\phi}{\rightarrow} C$, the equality $\textsf{M}_F(\phi\circ\psi) = \textsf{M}_F(\psi)\circ\textsf{M}_F(\phi)$ follows from the observation that for prime ideals $P\subset P^{\prime}$ implies $\mu(P) = \mu(P^{\prime})$ combined with $(\phi\circ\psi)^{-1}(P) \subset \psi^{-1}(\mu_B(\phi^{-1}(P)))$.

\vspace{2mm}

 By \rom{4} Lemma~\ref{lem : propCAlg} we know that $C(X,F)$ is a $pm$-ring if $X\in\textsf{KH}_F$, therefore, we can restrict $\textsf{M}_F$ to the category $\textsf{CAlg}_F$, and we obtain the maximal spectrum functor for algebras of continuous functions as before. 

\vspace{2mm}

\section{Gelfand Algebras and their Spectra}

\label{sec : 2}

The topological side of Gelfand duality is well understood (Theorem~\ref{thm: TGK}) and in most cases the category $\textsf{KH}_F$ is easily calculated (Lemma~\ref{thm:KH}). On the other hand, the category $\textsf{CAlg}_F$ has only been characterized, so far, for $\mathbb{R}$, $\mathbb{C}$, and complete fields with a non-archimedean absolute value. Where we have the famous Gelfand-Naimark theorems \cite{Rudin} and Van der Put theorem \cite{VDP}. 

This actually comprehends all absolute value (complete) cases. Since by Ostrowski theorem \cite{Witold} if the absolute value is archimedean and the field $F$ is complete then $F=\mathbb{R}$ or $F = \mathbb{C}$. 

Now, if we want to characterize the category $\textsf{CAlg}_F$ for an arbitrary topological field, we know some necessary conditions (Lemma~\ref{lem : propCAlg}), the Gelfand property being the most important one. Given a Gelfand algebra $A$ we can define the evaluation functions $ev_a: \textnormal{Max}(A)\to F$ by $ev_a(M) = a+M$. However, there is no reason to believe that these functions are Zariski continuous. 

Also, the proofs of the Gelfand-Naimark theorem and Van der Put theorem rely heavily on the metric structure of the algebras. This in contrast to Section~\ref{sec : 1}, where we make no mention of any metric/uniform structure. 

To summarize, there are two main questions to address:

\begin{enumerate}[label = \enum]
    \item \emph{When are the evaluation functions $\uptau_Z$-continuous?}
    \item \emph{What is the role of the metric structure?}
\end{enumerate}

For the first question, if we look at a classical proof \cite{Rudin}, the maximal spectrum is endowed, using the Gelfand property, with a different topology making the evaluation functions continuous. This topology is finer than the Zariski topology but, they coincide under rather natural suppositions.

For the second one, we will show that the metric/uniform structure arises canonically from the algebra (Definition~\ref{def : nunif}). On top of that, the uniform structure does not alter the algebraic information, more formally, see Theorem~\ref{thm : UnifCont}. 
 
\subsection{The Gelfand topology}
 
Let $A$ be a Gelfand algebra. The \emph{Gelfand topology} over Max($A$) is the weakest topology that makes the evaluations $ev_a$ continuous for all $a\in A$. We denote it by $\uptau_G$.

An equivalent way to describe the Gelfand topology is by seeing $\textnormal{Max}(A)\subset F^A$ as a subspace with the product topology under the identification $M\mapsto \pi_M$, where $\pi_M: A \to A/M = F$ is the canonical quotient map. This topology is finer than the Zariski topology ($\uptau_Z\leq \uptau_G$) since $D(a)=ev_a^{-1}(F^\times)$, and it is clearly Hausdorff.

\begin{lem}
    Let $A$ be a Gelfand $F$-algebra. Then $(\textnormal{Max}(A),\uptau_G)$ is closed in $F^A$ (product topology).
    \label{lem: MAX closed}
\end{lem}

\begin{proof}
    Note that, identifying the elements of $\textnormal{Max}(A)$ with the corresponding morphism $f:A\to F$, $f\in \textnormal{Max}(A)\subset F^A$ iff $f(\alpha x+ \beta y) = \alpha f(x)+ \beta f(y)$, $f(xy)=f(x)f(y)$ and $f(1)=1$ for all $x,y \in A$ and all $\alpha,\beta \in F$. Let $(\pi_{M_\lambda})_\lambda\subset\textnormal{Max}(A)$ be a net convergent to $f\in F^A$. Then $\pi_{M_\lambda}(a)\to f(a)$ for all $a\in A$, and due to the continuity of sum and multiplication we have
    \[
    \pi_{M_\lambda}(\alpha x+\beta y) = \alpha\pi_{M_\lambda}(x) + \beta \pi_{M_\lambda}(y) \to \alpha f(x) + \beta f(y)
    \]
    from which follows $f(\alpha x+\beta y)=\alpha f(x)+\beta f(y)$. Analogously $f(xy)=f(x)f(y)$ and $f(1)=1$.
\end{proof}

\vspace{2mm}

\noindent As an observation, since $ev_a(\textnormal{Max}(A)) = \sigma(a)$, we have $\textnormal{Max}(A)\subset \prod_{a\in A}\sigma(a)$.

\begin{lem}
    Let $A$ be a Gelfand algebra. Then $\uptau_G = \uptau_Z$ iff for all $\uptau_G$-closed $C\subset \textnormal{Max}(A)$ and $M\not\in C$ there exists $a\in A$ such that $ev_a(C) = 0$ and $ev_a(M)=1$.
    \label{lem: Z=G}
\end{lem}

\begin{proof}
    "$\Rightarrow$" Let $I\subset A$ be an ideal and $M\not\in V_I$ then there exists $a\in I\setminus M$, it follows $ev_a(V_I) = 0$ and $ev_a(M) \neq 0$. 
    
    "$\Leftarrow$" Let $C\subset \textnormal{Max}(A)$ $\uptau_G$-closed, we'll see that $C=V_I$ where $I= \bigcap_{M\in C}M$. Assume there exists $M\not\in C$ such that $I\subset M$, by hypothesis it exists $a\in A$ such that $ev_a(C) = 0$ and $ev_a(M) = 1$. This implies that $a\in I$ but this is a contradiction to the fact that $ev_a(M) = 1$. 
\end{proof}

\vspace{2mm}

Essentially, the previous Lemma tells us that the two topologies coincide exactly when the evaluation maps, $ev_a$, serve as witnesses for the Tychonoff property. Note that this makes sense, since, if the duality holds, we have $I_F(A) = C(\textnormal{Max}(A),F)$ and $(\textnormal{Max}(A),\uptau_Z)$ is $F$-Tychonoff. 

\begin{cor}
    Let $X$ be a compact $F$-Tychonoff space, then the Zariski and the Gelfand topology coincide for $\textnormal{Max}(C(X,F))$.
\end{cor}

\begin{proof}
It follows from the fact that $G_F$ is a homeomorphism and the fact that the diagram
    \begin{center}
     \begin{tikzpicture}
    \matrix (m) [matrix of math nodes, name=m ,row sep=1.5mm,column sep=0.5mm,minimum width=0.5mm]
    {
    & \textnormal{Max}(C(X,F)) && \\
    &&& \hspace{1.5cm} f\circ G_F^{-1}=ev_f \\
    X && F & \\
    };
    \path[-stealth]
    (m-1-2) edge node[above, rotate=45]{$G_F^{-1}$} (m-3-1)
    (m-3-1) edge node[below]{$f$} (m-3-3)
    (m-1-2) edge node[above, rotate=-45]{$ev_f$} (m-3-3)
    ;
    \end{tikzpicture}
\end{center}
is commutative.
\end{proof}

\begin{prop}
    Let $A$ be a Gelfand $F$-algebra. Then $(\textnormal{Max}(A),\uptau_G)$ is compact iff $\sigma(a)$ is compact for all $a\in A$.
    \label{prop: MaxCom}
\end{prop}

\begin{proof}
    If $\uptau_G$ is compact then $\sigma(a) = ev_a(\textnormal{Max}(A))$ is compact. If $\sigma(a)$ is compact for all $a\in A$ then $\prod_{a\in A}\sigma(a)$ is compact, by Lemma~\ref{lem: MAX closed} it follows that $\uptau_G$ is compact.
\end{proof}

\begin{cor}
    Assume that $\sigma(a)$ is compact for all $a\in A$. The following are equivalent
    \begin{enumerate}[label = \enum]
        \item $\uptau_Z = \uptau_G$.
        \item $\uptau_Z$ is Hausdorff.
    \end{enumerate}
    \label{cor: espectC Z=G}
\end{cor}

\subsection{Canonical Uniformity over a Gelfand Algebra}

 Let $F$ be a topological field a \emph{base of symmetric neighborhoods of zero} is a collection $\mathcal{B}$ of open neighborhoods of $0$ satisfying:

\begin{enumerate}[label=\enum]
    \item $\bigcap \mathcal{B}=\{0\}$. 
    \item $\forall U\in \mathcal{B} $ we have $ U = -U$.
    \item $\forall U,V\in\mathcal{B} \hspace{1mm} \exists  W\in\mathcal{B} $ such that $ W\subset U\cap V$. 
    \item $\forall U\in\mathcal{B} \hspace{1mm} \exists V\in\mathcal{B} $ such that $ V+V \subset U$.
    \item $\forall U \in \mathcal{B} \hspace{1mm} \exists V\in\mathcal{B}$ such that $VV\subset U$.
    \item $\forall U\in\mathcal{B} \hspace{1mm} \forall x\in K \hspace{1mm} \exists V\in \mathcal{B}$ such that $xV\subset U$.
    \item $\forall U\in\mathcal{B} \hspace{1mm} \exists V\in\mathcal{B}$ such that $(1+V)^{-1}\subset 1+U$.
\end{enumerate}

\noindent Such a collection always exists, and the translations $a+U$ with $a\in F$ and $U\in\mathcal{B}$ form a basis for the topology \cite{Witold}. 

\begin{defn}
    Let $A$ be a Gelfand algebra. Define the \emph{canonical uniformity} over $A$ by taking the subsets
    \begin{align*}
        V_U &= \{(a,b)\in A\times A: (b-a)+M\in U \hspace{1mm} \forall M\in \textnormal{Max}(A) \} \\
        &=\{(a,b)\in A\times A: \sigma(b-a)\subset U\}
    \end{align*}
    with $U\in\mathcal{B}$, where $\mathcal{B}$ is a base of symmetric neighborhoods of zero, as basic entourages. 
    \label{def : nunif}
\end{defn}

It is easy to verify that $\{V_U\}_{U\in\mathcal{B}}$ indeed constitutes a basis for a uniformity and it does not depend on the choice of a base of symmetric neighborhoods of zero. 


In the case $A=C(X,F)$ with $X$ compact Hausdorff, we have $\sigma(f) = f(X)$. Since $F$ is a uniform space it follows that the induced uniform topology is the topology of uniform convergence, which is equivalent to the compact-open topology, the proof of this fact follows the same steps as in (\cite{Dugundji}, Ch 12, Theorem 7.2).

Recall that for uniform spaces $X$ and $Y$ the space of uniformly continuous functions from $X$ to $Y$ is a uniform space, which is complete whenever $Y$ is complete (\cite{Isbell}, Ch 3, Theorem 31). Moreover, if $X$ is compact then the space of uniformly continuous functions coincides with $C(X,Y)$ (with the compact-open topology). Therefore:

\begin{thm}
    Let $F$ be a topological field complete for its group uniformity. Then $C(X,F)$ is complete with respect to the canonical uniformity.
    \label{cor : Ccomp}
\end{thm}

Now, if $(F,|\cdot|)$ is a normed field and $A$ a $F$-algebra, the \emph{spectral radius} of $a\in A$ is $\rho(a)= \sup\{|\lambda|: \lambda\in\sigma(a)\}$. Let $X$ be a compact Hausdorff space, then $C(X,F)$ is normable, to be more precise we can define the sup norm, $\norm{f}_\infty = \sup\{|f(x)|:x\in X\}$. Notice that, since $G_F$ is surjective, we have:

\begin{align*}
    \norm{f}_\infty &= \sup \hspace{0,5mm} \{|f(x)|:x\in X\} \\
    &= \sup \hspace{0,5mm} \{|\lambda|: \lambda\in\sigma(f)\} \\
    &= \rho(f)
\end{align*}

On the other hand, given a normed Gelfand $F$-algebra $A$, it follows from the inequality $|(y-x) + M|\leq \norm{y-x}$ that $V_{U_{\epsilon}} \supset V_\epsilon$, where $V_{U_\epsilon}=\{(x,y): \sigma(x-y)\subset U_\epsilon\}$ and $V_\epsilon =\{(x,y): \norm{x-y}<\epsilon\}$ are the basic entourages of the natural and metric uniformity, respectively. To summarize, the Gelfand map (defined below), if $(\textnormal{Max}(A),\uptau_G)$ is compact, is norm decreasing and is an isometry exactly when the norm and the spectral radius coincide, moreover, the spectral radius is a (semi)norm inducing the canonical uniformity. 

\begin{defn}
Let $A$ be a Gelfand $F$-algebra. Considering the maximal spectrum with the Gelfand topology define the \emph{Gelfand map} as
\[
I_F: A\to C(\textnormal{Max}(A),F), \hspace{3mm} I_F(a) = ev_a
\]  
\label{def:GelMap}
\end{defn}

\begin{cor}
    Any Gelfand semisimple $F$-algebra may be represented as an algebra of continuous functions. 
\end{cor}

\noindent The following theorem is a direct consequence of the definition.

\begin{thm}
    Let $A$ be a Gelfand $F$-algebra then
    \begin{enumerate}[label = \enum]
        \item $I_F$ is a homomorphism of algebras, its kernel is $\textnormal{Jrad}(A)$, thus it is injective iff $A$ is semisimple.
        \item $I_F$ is uniformly continuous. 
    \end{enumerate}
    \label{thm: GelMap}
\end{thm}

Notice that for a Gelfand $F$-algebra $A$ such that $(\textnormal{Max}(A),\uptau_G)$ is compact if $I_F$ is injective then it is also a uniform embedding; that is, $A$ and $I_F(A)$ are uniformly isomorphic, because for $a,b\in A$ we have $\sigma(a-b) = ev_{a-b}(\textnormal{Max}(A)) = \sigma(I_F(a)-I_F(b))$, thus for any $U\subset F$ symmetric neighborhood of $0$ we have 

\[
I_F\times I_F(V_U) = \{(ev_a,ev_b)\in  C(\textnormal{Max}(A),F)\times C(\textnormal{Max}(A),F): \sigma(ev_{a-b}) \subset U\} 
\]

\noindent with $V_U\subset A\times A$ as in Definition~\ref{def : nunif}. The set above is exactly a basic entourage for the subspace uniformity over $I_F(A)$, from which follows that $I_F^{-1}$ is uniformly continuous. 

Now, even if $(\textnormal{Max}(A),\uptau_G)$ is not compact, if $I_F$ is injective then is going to be an algebraic isomorphism between $A$ and its image.


The canonical uniformity captures some of the algebraic information of the algebra, for instance, it is not hard to prove that $A$ is semisimple iff the canonical uniformity is separated. Moreover,

\begin{prop}
    Let $A$ be a Gelfand $F$-algebra. Then $U(A)$, the set of units of  $A$, is open (in the uniform topology) iff $\sigma(a)$ is closed (in $F$) for all $a\in A$.
\end{prop}

\begin{proof}
    Suppose $U(A)$ is open then $F\setminus\sigma(a) = \{u\in A: a-u\in U(A)\}\cap F = (a+U(A))\cap F$ is open, since $a+U(A)$ is open in $A$. Now, assume that $\sigma(a)$ is closed for all $a\in A$. Let $a\in U(A)$, then $0\not\in \sigma(a)$. Take $U$ a neighborhood of $0$ such that $\sigma(a)\cap U = \emptyset$ and consider the neighborhood $V_U[a]$ of $a$. If $b\in V_U[a]$ then $b-a + M = \lambda \in U$ then $\lambda\not\in \sigma(a)$. Since $b+M = a - (a-b) + M \neq 0$  for all  $M\in \textnormal{Max}(A)$ we have that $b$ is invertible, therefore, $V_U[a]\subset U(A)$.
\end{proof}

\begin{prop}
    Let $A$ be a Gelfand $F$-algebra. If $A$ is semisimple then $F$ is closed in $A$.
\end{prop}

\begin{proof}
  Let $(a_\gamma)_\gamma\subset F$ a convergent net to $a\in A$. Let $W,U\subset F$ a neighborhoods of $0$ such that $W + W \subset U$ then $a_\gamma\in V_W[a]$, hence, $a-a_\gamma + M \in W$ for all $M\in\textnormal{Max}(A)$. Let $\lambda_M = a+M $, for all $M,N \in \textnormal{Max}(A)$ we have
  \[
  \lambda_M - \lambda_N = ( \lambda_M - a_\gamma )  - ( \lambda_N - a_\gamma) \in W + W \subset U,
  \]
  therefore, $\lambda_M = \lambda_N = \lambda$. By the previous argument, we have $\sigma(a) = \{\lambda\}$, since $A$ is semisimple it follows $a=\lambda \in F$.
\end{proof}

\begin{cor}
Let $X$ be a compact Hausdorff spaces then $F$ is closed in $C(X,F)$ and $U(C(X,F))$ is open.
\end{cor}

Recall that in a Gelfand algebra any maximal ideal could be regarded as both a subset of the algebra or as an algebra morphism to the field. From this it follows that the preimage of a maximal ideal is a maximal ideal. Given a map $\phi:A\to B$ between Gelfand algebras and a maximal ideal $M\subset B$ it is easy to verify that $\phi^{-1}(M) = \textnormal{ker}(\pi_M\circ\phi)$, where $\pi_M; B\to F$ is the canonical projection.

\begin{thm}
    Let $A,B$ Gelfand $F$-algebras and $\phi: A\to B$ a $F$-algebra morphism. Then $\phi$ is uniformly continuous.
    \label{thm : UnifCont}
\end{thm}

\begin{proof}
    Let $U\subset F$ a symmetric neighborhood of $0$. We denote by $V_U^A, V_U^B$ the basic entourages of the canonical uniformity over $A$ and $B$, respectively. Let $(a,b)\in V_U^A$, then for all $M\in\textnormal{Max}(A)$ we have $(a-b) + M \in U$, in particular, $(a-b)+\phi^{-1}(N)$ for all $N\in\textnormal{Max}(B)$. Since $\phi^{-1}(N) = \textnormal{ker}(\pi_N\circ\phi)$ it follows that $\phi(a)-\phi(b) + N \in U$, therefore, $(\phi\times\phi)(V_U^A) \subset V_U^B$.
\end{proof}

\vspace{2mm}

\noindent It follows from the remark after Definition~\ref{def:GelMap} that:
 
\begin{cor}

 Let $f:A\to B$ be a $*$-morphism of $C^*$-algebras. Then $f$  is continuous and norm decreasing.
    
\end{cor}

As we have seen Gelfand algebras can be endowed with a canonical uniform structure, moreover, this uniformity interacts pretty well with the algebraic structure as we can observe in the previous results. Except for Theorem~\ref{thm : UnifCont}, which is a generalization of a well-known fact about $C^*$-algebras, the previous propositions are basic results in Banach complex algebras. Recall by the Gelfand-Mazur theorem any (unital) complex Banach algebra is a Gelfand algebra. 

The importance of Theorem~\ref{thm : UnifCont} cannot be stressed enough. It shows that the category of Gelfand algebras and, consequently the categories $\textsf{CAlg}_F$, remains unchanged if we regard their objects as uniform structures. This will become important in Section~\ref{sec : 4} where we give a characterization of the algebras of continuous functions for a disconnected field.

\section{Gelfand representation for complete disconnected fields}

\label{sec : 4}

In this section, we generalize Van der Put theorem to any complete disconnected topological field. Here, completeness is with respect to the canonical group uniformity. As a consequence, we get a full description of the category $\textsf{CAlg}_F$ for such fields. 

Furthermore, in the process we show that everything could be built up from the algebra alone. This in contrast to the classical proofs which rely on famous theorems from analysis, namely, Stone-Weierstrass theorem and Alaoglu theorem.

The idempotent elements of a ring $R$ are the elements $a\in R$ such that $a^2=a$. The set of idempotents of a ring, $B(R)$, forms a boolean algebra with the operations $a^*=1-a$, $a\wedge b= ab$, and $a\vee b= a+b-ab$, $0,1$ being, respectively, the top and bottom elements. Moreover, given $a,b\in B(R)$ we have the following relations: $\textnormal{Max}(R)\setminus D(a)=D(1-a)$, $D(a)\cap D(b)=D(a)\wedge D(b)$ and $D(a)\cup D(b)=D(a\vee b)$, for the last one consider $D(1-a)\cap D(1-b)= D((1-a)(1-b))=D(1-a\vee b)$. 

\begin{prop}
Let $X$ be a Stone space and $F$ a disconnected field. Then $\langle B(C(X,F))\rangle_F$, the $F$-vector space generated by the idempotents, is uniformly dense in $C(X,F)$.
\label{prop : CharD}
\end{prop}

\begin{proof}
    Let $f\in C(X,F)$ and $U\subset F$ a neighborhood of $0$. For each $x\in X$ choose a clopen neighborhood $C_x$ of $x$ such that $f(C_x)\subset f(x)+U$. Since $X$ is compact there exist $C_{x_1},...,C_{x_n}$ covering $X$, and without lost of generality we can assume they are disjoint. Define $g_U = \sum_{i=1}^n f(x_i)\chi_{C_{x_i}}$, where $\chi_{C_{x_i}}$ is the characteristic function of $C_{x_i}$, then $(g_U-f)(X) \subset U$; equivalently, $(f(x),g_U(x))\in V_U$ for all $x\in X$.
\end{proof}

\begin{prop}
    Let $A$ be  a $F$-algebra, then, $a\in \langle B(A) \rangle_F$
    \begin{enumerate}[label=\enum]
        \item $a$ is a linear combination of orthogonal idempotents; i.e., $a=\sum_{i=1}^n\lambda_ie_i$  with $\lambda_i\in F$, $e_i\in B(A)$ and $e_ie_j=0$ for $i\neq j$.
        \item The elements of $\sigma(a)$ are precisely the $\lambda_i's$, plus $0$ if $a$ is not invertible.
        \item For all $M\in \textnormal{Max}(A)$ we have $\pi_M(\langle B(A)\rangle_F)=F$, where $\pi_M: A\to A/M$ is the canonical projection.
        \item For every $F$-algebra $A$ the subspace spanned by the idempotents, $\langle B(A) \rangle_F$, is a Gelfand algebra.
    \end{enumerate}
    \label{prop: IdemOrt}
\end{prop}

\begin{proof}
    \rom{1} is proven in \cite{Dominguez}. \rom{2} Note that if $a$ is an idempotent other than $0$ and $1$ then $\sigma(a)=\{0,1\}$. Since $a$ is neither $0$ nor $1$ we have $a(a-1)=0$ and $0,1\in \sigma(a)$. Now, if $\lambda\in\sigma (a)$ then $\lambda^2-\lambda\in\sigma (a^2-a)=\sigma(0)=\{0\}$ and $\lambda(\lambda-1)=0$. Given $a=\sum_{i=1}^n \lambda_i e_i$ with $e_i$ orthogonal idempotents and $\lambda_i\in F$, let $M\in \textnormal{Max}(A)$ if $a\in M$ then $0\in \sigma(a)$, if not, then there exists $e_i\not\in M$ and since $e_ie_j=0$ for $i\neq j$ then $e_j\in M$, and we have $a+M= \lambda_i e_i + M = \lambda_i +M$. \rom{3} and \rom{4} follow immediately from \rom{2}.
\end{proof}

\vspace{2mm}

The following Lemma generalizes (Prop 2, \cite{DominguezGomez1981}) to the setting of Gelfand algebras equipped with their canonical uniform structure.

\begin{lem}
    Let $A$ be a Gelfand algebra such that $\langle B(A) \rangle_F$ is uniformly dense, then every maximal ideal is determined by its idempotents.
    \label{lem : MaxDet}
\end{lem}

\begin{proof}
    We will see that in the uniform topology $\overline{\langle B(A)\cap M \rangle}_F= M$, equivalently, for all $x\in M$ and  for all $U\in \mathcal{B}$ there exists $e\in \langle B(A)\cap M \rangle_F$ such that $(x,e)\in V_U$.

    Let $M\in \textnormal{Max}(A)$, $x\in M$ and $W,U\in\mathcal{B}$ such that $W+W\subset U$. Since $\langle B(A) \rangle_F$ is dense there exists $e\in \langle B(A) \rangle_F$ such that $(x,e)\in V_W$, i.e., $x-e + N\in W$ for all $N\in\textnormal{Max}(A)$. If $e\in M$ the proposition follows immediately, if $e\not\in M$ then by Proposition~\ref{prop: IdemOrt} $e=\sum_{i=1}^n\lambda_ie_i$ with $e_i\in B(A)$ orthogonal, by reorganizing we can assume $e_1\not\in M$ then $e_i\in M$ for $i\geq 2$, in particular, $e-\lambda_1e_1\in M$. This implies
    \[
    \lambda_1 =  \lambda_1e_1 + \sum_{i=2}^n \lambda_ie_i - x + M = e-x +M \in W.
    \]
    Note that for all $N\in \textnormal{Max}(A)$ we have
    \[
    x - (e-\lambda_1e_1) + N = (x-e) + \lambda_1 + N \in W+W \subset U
    \]
    then $(x,e-\lambda_1e_1)\in V_U$.
\end{proof}

\begin{cor}
    Let $A$ be a Gelfand algebra such that $\langle B(A)\rangle_F$ is uniformly dense then $(\textnormal{Max}(A),\uptau_Z)$ is Hausdorff.
    \label{cor : ZarHaus}
\end{cor}

\begin{proof}
    Let $M,M^\prime\in \textnormal{Max}(A)$ be different ideals, by Lemma~\ref{lem : MaxDet} there exists $e\in M\cap B(A)$ such that $e\not\in M^\prime$. Since $e(1-e)=0$ then $M\in D(1-e)$, $M^\prime\in D(e)$ and $D(e)\cap D(1-e)=\emptyset$.
\end{proof}

\vspace{2mm}

Recall a subset $A\subset F$ of a topological field is totally bounded if for all $U\in\mathcal{B}$ there exist $y_1,...,y_n\in K$   such that $A\subset \bigcup_{i=1}^n y_1 + U$.

\begin{prop}
    Let $A$ be a $F$-algebra. Let $x\in\overline{\langle B(A)\rangle_F}$ then $\sigma(x)$ is totally bounded.
    \label{prop : EspecBound}
\end{prop}

\begin{proof}
    Let $U\in\mathcal{B}$. Given $x\in \overline{\langle B(A) \rangle}_F$ there exists $e\in \langle B(A) \rangle_F$ such that $(x,e)\in V_U$, this implies that for all $\lambda \in \sigma(x)$ exists $\mu\in \sigma(e)$ such that $\lambda-\mu\in U$, from which follows $\sigma(x)\subset \bigcup_{\mu\in\sigma(e)}\mu+U$ and since $\sigma(e)$ is finite (Proposition~\ref{prop: IdemOrt} \rom{2}) we have that $\sigma(x)$ is totally bounded.
\end{proof}

\begin{cor}
    Let $F$ be a disconnected complete field and $A$ a Gelfand $F$-algebra such that $\langle B(A)\rangle_F$ is uniformly dense. Then $\uptau_Z=\uptau_G$ and $\textnormal{Max}(A)$ is a Stone space. 
    \label{cor : Z=G}
\end{cor}

\begin{proof}
    By Proposition~\ref{prop : EspecBound} we have that $\sigma(a)$ is totally bounded for all $a\in A$, hence $\overline{\sigma(a)}$ is compact (\cite{Willard}, Theorem 39.9). The proposition now follows from Lemma~\ref{lem: MAX closed} and Corollaries~\ref{cor: espectC Z=G} and \ref{cor : ZarHaus}.
\end{proof}

\begin{prop}
    Let $F$ be a topological disconnected complete field and $A$ a Gelfand $F$-algebra such that $\langle B(A)\rangle_F$ is uniformly dense in $A$. Then $\{D(e): e\in B(A)\}$ is a basis for the Zariski topology in $\textnormal{Max}(A)$. 
    \label{prop: IdemBase}
\end{prop}

\begin{proof}
    First notice that $\uptau_Z=\uptau_G$ (Corollary~\ref{cor : Z=G}). Since $\textnormal{Max}(A)$ is a Stone space (Corollary~\ref{cor : Z=G}) it suffices to show that for every clopen $C\subset\textnormal{Max}(A)$ and $M\in C$ there exists $e\in B(A)$ such that $M\in D(e)\subset C$. Take $N\in \textnormal{Max}(A)\setminus C$, in Corollary~\ref{cor : ZarHaus} we proved that there exists $e_N\in B(A)$ such that $e_N\in M$ and $e_N\not\in N$, then $\{D(e_N): N\in \textnormal{Max}(A)\setminus C\}$ is an open cover of $\textnormal{Max}(A)\setminus C$ and since this set is compact there exist $e_{N_1},...,e_{N_k}$ such that $\textnormal{Max}(A)\setminus C \subset \bigcup_{i=1}^kD(e_{N_i})$. Since no $D(e_N)$ contains $M$ we have
    \[
    M \in D(\prod_{i=1}^k (1-e_{N_i})) = \bigcup_{i=1}^k \textnormal{Max}(A) \setminus D(e_{N_i}) \subset C.
    \]
\end{proof}

We are now ready to present our generalization of the Gelfand-Naimark theorem for the complete disconnected case.

\begin{thm}
Let $F$ be a topological disconnected field complete with respect to its canonical uniformity. A $F$-algebra $A$ is the algebra of $F$-valued continuous functions over a Stone space iff $A$ is Gelfand, semisimple, complete, and $\langle B(A) \rangle_F$ is uniformly dense in $A$.
\label{thm:GND}
\end{thm}

\begin{proof}
    We have seen that $C(X,F)$ is Gelfand, semisimple and complete, and that the characteristic functions of clopen sets are uniformly dense. Now, if $A$ is Gelfand and semisimple then $I_F:A\to C(\textnormal{Max}(A),F)$ is injective and its image is closed since $A$ is complete and, by the remark after Theorem~\ref{thm: GelMap}, $I_F$ is a uniform embedding. 
    
    Given a $C\subset \textnormal{Max}(A)$ clopen, by Proposition~\ref{prop: IdemBase}, there exist $e_1,...e_n\in B(A)$ such that
    \[
    D(\bigvee_{i=1}^ne_i) = \bigcup_{i=1}^nD(e_n) = C
    \]
    it follows that $I_F(e)=\chi_C$, where $e = \bigvee_{i=1}^ne_i$ and that $I_F(A)$ is dense in $C(\textnormal{Max}(A),F)$ (Proposition~\ref{prop : CharD}), therefore, $I_F(A)= C(\textnormal{Max}(A),F)$.
\end{proof}

\vspace{2mm}

\begin{cor}[\cite{Dominguez}, Van der Put] Let $F$ be a non-archimedean normed, complete field. A non-archimedean Banach $F$-algebra $(A,\norm{\cdot})$ is isometrically isomorphic to the algebra of continuous functions over a Stone space iff $\langle \{e\in B(A): \norm{e} = 1\} \rangle_F$ is dense in $A$. 
\label{thm : VDP} 
\end{cor}

\begin{proof}
   One direction is trivial. On the other direction, it suffices to prove that under these assumptions $A$ is a Gelfand algebra and that the canonical uniformity and the metric uniformity coincide. Then we can apply Theorem~\ref{thm:GND}.

    Denote by $E= \langle \{e\in B(A): \norm{e} = 1\} \rangle_F$. Using the fact that $\pi_M(E) = F$, $\pi_M(E)$ is dense (since $E$ is) and $F$ is closed in $A/M$ (since $F$ is complete) we have that $A/M=F$ for all $M\in \textnormal{Max}(A)$. Now, by (\cite{Dominguez}, Theorem 1) $\sigma(a)$ is totally bounded, hence,  $(\textnormal{Max}(A),\uptau_G)$ is compact. 

    Finally, that the canonical uniformity and the metric uniformity coincide is equivalent to the Gelfand map being an isometry. Note that for $a\in A$ we have:
    \begin{align*}
    \norm{I_F(a)}_\infty &= \hspace{0,5mm} \sup\{|I_F(a)(M)|: M\in \textnormal{Max}(A)\} \\
    &= \hspace{0,5mm} \sup \{|\lambda|: \lambda\in \sigma(a)\} \\ 
    &= \rho(a)
\end{align*}
Take $e\in E$, by Proposition~\ref{prop: IdemOrt} we have $e = \sum_{i=1}^n \lambda_ie_i$ with $e_i\in I(A)$ ortogonal, we can also guarantee that they have norm $1$ (\cite{Dominguez}). Then
\[
\rho(e) \leq \norm{e} \leq \max\{\lambda_i\} = \rho(e)
\]
from which follows $\norm{e} = \rho(e)$. Given $a\in A$ by hypothesis there exists $(e_n)_n\subset E$ such that $e_n\to a$, since $I_F$ is continuous 
\[
\norm{I_F(a)}_\infty = \lim_{n\to\infty} \norm{I_F(e_n)}_\infty = \lim_{n\to\infty} \rho(e_n) = \lim_{n\to\infty} \norm{e_n} = \norm{a}
\]
\end{proof}

Notice that the only step that involves the completeness of the field in the proof is the surjectivity of the Gelfand map. So, if we try to extend the result to non-complete fields, for instance $\mathbb{Q}$, we do have that for any $F$-algebra $A$ that satisfies the conditions of Theorem~\ref{thm:GND}, except maybe for the completeness,  $I_F$ is epi, since its image is (uniformly) dense, but this is not a sufficient condition for $I_F$ to be surjective.



\section{General Gelfand Adjunction}

\label{sec : 3}

In this section we give an explicit proof of Gelfand duality, showing in the process that there is a (dual) adjunction between $\textsf{KH}_F$ and the category of $F$-algebras satisfying \rom{1},\rom{2} and \rom{4} from Lemma~\ref{lem : propCAlg}, which extends Gelfand duality. Also we show that if the field is complete and satisfies the Stone-Weierstrass theorem, conditions in Lemma~\ref{lem : propCAlg} plus completeness characterize the algebras of continuous functions. 

In Lemma~\ref{lem : propCAlg} we proved that $C(X,F)$ for a compact $F$-Tychonoff space $X$ has the following properties: \rom{1} Gelfand, \rom{2} semisimple, \rom{3} the spectrum of each element is compact and \rom{4} every prime ideal is contained in a unique maximal ideal. All of these properties appear explicitly throughout, except maybe for \rom{4} which plays a more subtle role.

Recall from Section~\ref{sec : 1} that the maximal spectrum is functorial over the category of Gelfand algebras. Unfortunately, the Gelfand property does not guarantee the Hausdorffness of $\uptau_Z$, this is where property \rom{4} plays a role since the $pm$-property implies that the maximal spectrum is Hausdorff. In fact, if the algebra is semisimple then the maximal spectrum is Hausdorff if and only if the algebra is a $pm$-ring (\cite{DominguezGomez1981}, Prop 1). 

If in addition \rom{3} holds then, the compactness of the spectrum of each element implies $\uptau_G$ is compact (Proposition~\ref{prop: MaxCom}) and the $pm$-property implies $\uptau_Z$ is Hausdorff, therefore, $\uptau_G=\uptau_Z$. Now, it follows from Lemma~\ref{lem: Z=G} that $\textsf{M}_FA \in \textsf{KH}_F$. 

Denote by $\textsf{A}_F$ the full subcategory of $\textsf{Alg}_F$ of algebras satisfying \rom{1}, \rom{3} and \rom{4}. Under these assumptions we have that the Gelfand transform and the Gelfand map define natural transformations between $G_F: \textsf{1}_{\textsf{KH}_F}\to \textsf{M}_F\textsf{C}_F$ and 
$I_F : \textsf{1}_{A}\to \textsf{C}_F\textsf{M}_F$, respectively. 

\[
\begin{matrix}
   G_X = G_F(X): & X & \to & \textnormal{Max}(C(X,F)) & & I_A = I_F(A) : & A & \to & C(\textnormal{Max}(A),F) \\
    & x & \mapsto & M_x & & & a & \mapsto & ev_a 
\end{matrix}
\]

 \noindent Recall that for $f:X\to Y$ and $\varphi:A\to B$ one has $\textsf{C}_Ff = f^* = -\circ f$ and $\textsf{M}_F\varphi = \varphi^{-1}$. The commutativity of the diagrams:

\begin{center}
    \begin{tikzpicture}
            \matrix (m) [matrix of math nodes, name = m, row sep = 1.5cm, column sep = 1cm]
            {
            X & Y & & A & B \\
        \textsf{M}_F\textsf{C}_F(X) & \textsf{M}_F\textsf{C}_F(Y) & & \textsf{M}_F\textsf{C}_F(A) & \textsf{M}_F\textsf{C}_F(B) \\
            };
            \path[-stealth]
            (m-1-1) edge node[above]{$f$} (m-1-2)
            (m-1-1) edge node[left]{$G_X$} (m-2-1)
            (m-1-2) edge node[right]{$G_Y$} (m-2-2)
            (m-2-1) edge node[below]{$(f^*)^{-1}$} (m-2-2)
            (m-1-4) edge node[above]{$\varphi$} (m-1-5)
            (m-1-4) edge node[left]{$I_A$} (m-2-4)
            (m-2-4) edge node[below]{$-\circ\varphi^{-1}$} (m-2-5)
            (m-1-5) edge node[right]{$I_B$} (m-2-5)
            ;   
    \end{tikzpicture}
\end{center}

\noindent follows from 

\[
G_Y\circ f(x) = M_{f(x)} = \{h\in C(Y,F): h(f(x)) = 0\} = \textnormal{ker}(ev_x\circ f^*) = (f^*)^{-1}\circ G_X(x) \\
\]
\noindent and
\[
ev_a\circ \varphi^{-1}(M) = ev_a(\textnormal{ker}(\pi_M\circ\varphi)) = \pi_M(\varphi(a)) = ev_{\varphi(a)}(M) 
\]
\noindent respectively.

\begin{prop}
    There is a dual adjunction between the category $\textsf{KH}_F$ and the category \textsf{A} of $F$-algebras satisfying \rom{1}, \rom{3} and \rom{4} of Lemma~\ref{lem : propCAlg}. 

    \label{thm : GA}
\end{prop}

\begin{proof}
By discussion above it only remains to show that the following diagrams commute (for simplicity we omit the subscript $F$)
    \begin{center}
    \begin{tikzpicture}
        \matrix (m) [matrix of math nodes, name=m ,row sep=10mm,column sep=10mm,minimum width=1.5mm]
        {
        \textsf{C} & \textsf{C}\textsf{M}\textsf{C} & & \textsf{M} & \textsf{M}\textsf{C}\textsf{M} \\
        & \textsf{C} & & & \textsf{M} \\
        };
        \path[-stealth]
        (m-1-1) edge node[above]{{\footnotesize $I$}\textsf{C}} (m-1-2)
        (m-1-2) edge node[right]{\textsf{C}{\footnotesize $G$}} (m-2-2)
        (m-1-4) edge node[above]{{\footnotesize $G$}\textsf{M}} (m-1-5)
        (m-1-5) edge node[right]{\textsf{M}{\footnotesize $I$}} (m-2-5)
        ;
        \path[]
        (m-1-1) edge [double] node {} (m-2-2)
        (m-1-4) edge [double] node {} (m-2-5)
        ;
    \end{tikzpicture}
\end{center}
For the diagram on the left, let $X$ be a $F$-Tychonoff space and notice that 
\[
ev_f\circ G(x) = ev_f(M_x) = f(x)
\]
 for the diagram on the right, let $A\in \textsf{A}$ and $N\in \textnormal{Max}(A)$, then $a\in I^{-1}(M_N)$ iff $ev_a\in M_N$ iff $ev_a(N) = 0$ iff $a\in N$, here $M_N\subset C(\textnormal{Max}(A),F)$ is the maximal ideal of functions that vanish at $N$. 
\end{proof}

\vspace{2mm}

We can restrict the adjunction of the previous theorem to the category $\textsf{CAlg}_F$. Turns out that this restriction gives us the general version of Gelfand duality. 

\begin{thm}[\cite{Vechtomov}, Proposition 8.2]
    The category $\textsf{KH}_F$ is dual to the category $\textsf{CAlg}_F$.

    \label{thm : GD}
\end{thm}

\begin{proof}
    Since any algebra $A$ in $\textsf{CAlg}_F$ is semisimple we just have to show that the Gelfand map $I_A : A\to C(\textnormal{Max}(A),F)$ is surjective. By definition there is an isomorphism $\alpha: A\to C(X,F)$ for some $X$ in $\textsf{KH}_F$, we will see that there is an isomorphism $\beta : C(X,F) \to C(\textnormal{Max}(A),F)$ such that the diagram commutes
    \begin{center}
        \begin{tikzpicture}
            \matrix (m) [matrix of math nodes, name=m ,row sep=10mm,column sep=10mm,minimum width=1.5mm]
            {
            A & \\
            C(X,F) & C(\textnormal{Max}(A),F) \\
            };
            \path[-stealth]
            (m-1-1) edge node[above]{$I_A$} (m-2-2)
            (m-1-1) edge node[left]{$\alpha$} (m-2-1)
            (m-2-1) edge node[below]{$\beta$} (m-2-2)
            ;
        \end{tikzpicture}
    \end{center}
    \noindent Consider the map $\textnormal{Max}(A)\overset{\alpha^{\prime\prime}}{\to} \textnormal{Max}(C(X,F)) \overset{G_X^{-1}}{\to} X$, where $\alpha^{\prime\prime}$ denotes the direct image. Since $\textsf{M}_F \alpha = \alpha^{-1}$ we have that $(\textsf{M}_F \alpha)^{-1} = \alpha^{\prime\prime}$. 
    
    Take $\beta = \textsf{C}_F(G_X^{-1}\circ\alpha^{\prime\prime}) = -\circ(G_X^{-1}\circ\alpha^{\prime\prime})$. Now, given $a\in A$ and $M\in\textnormal{Max}(A)$ we have
    \begin{align*}
        \alpha(a)\circ(G_X^{-1}\circ\alpha^{\prime\prime})(M) &= \alpha(a)(G_X^{-1}(\alpha(M))) \\
        &= \alpha(a) + \alpha(M)
    \end{align*}
    \noindent since $\alpha$ preserves the unity, and thus it preserves the field, we have that if $\lambda - a \in M$ (therefore $\lambda = a+ M$), then $\alpha(a)-\lambda \in \alpha(M)$, from which follows:
    \[
    ev_a(M) = \lambda =   \alpha(a)\circ(G_X^{-1}\circ\alpha^{\prime\prime})(M)
    \]
\end{proof} 

\vspace{2mm}

Classically Gelfand duality involves the norm of the algebras, in Section~\ref{sec : 2} we showed that the uniform structure induced by the norm is in fact an algebraic invariant of Gelfand algebras. Since objects of $\textsf{CAlg}_F$ (or $\textsf{A}_F$) can also be regarded as uniform spaces and, by Theorem~\ref{thm : UnifCont}, this does not alter the category we have the following corollary. 

\begin{cor}
    There is a dual adjunction between the category $\textsf{KH}_F$ and the category $\textsf{A}_F$ of $F$-algebras satisfying \rom{1}, \rom{3} and \rom{4} of Lemma~\ref{lem : propCAlg} endowed with the canonical uniformity. Furthermore, if we restrict ourselves to $\textsf{CAlg}_F$ the adjunction restricts to a duality. 
\end{cor}

By the Gelfand-Kolmogorov theorem $G_F$ is a natural isomorphism (Theorem~\ref{thm: TGK}), therefore, $\textsf{KH}_F$ is actually a retract of $\textsf{A}_F$. This shows that we are not that far from getting a duality, but \emph{how far are we from it?}

Notice that the only thing separating the (dual) adjunction from the duality is the Gelfand map. By Theorem~\ref{thm: GelMap}  $I_F$ is injective iff the algebra is semisimple, so $I_F$ being monomorphic is completely characterized. 

The surjectivity of $I_F$, in the general case, is an open problem. Classically, this is achieved by observing first that, under the suppositions over the norm plus completeness, $I_F$ is an isometry, therefore its image is closed, and then, using the Stone-Weierstrass theorem, show that its image is also dense. Combining these two facts we get that $I_F$ is surjective. 

A topological field $F$, different from $\mathbb{C}$, is said to satisfy the Stone-Weierstrass theorem if, for every $X$ compact Hausdorff, any subalgebra $A\subset C(X,F)$ that contains the constant functions and separates points is uniformly dense (canonical uniformity)\footnote{Recall that in this case, the uniform topology coincides with the compact-open topology.}.  In the complex case, we additionally require the subalgebra $A$ be closed under conjugation. 

Let $F$ be a complete field that satisfies the Stone-Weierstrass theorem and let $A$ be an $F$-algebra satisfying conditions \rom{1}-\rom{4} and complete (w.r.t. the canonical uniformity). By the remark after Theorem~\ref{thm: GelMap} the semisimplicity and the completeness imply that the image of $I_F$ is closed, moreover, by Lemma~\ref{lem: Z=G}  $I_F(A)$ separates points, therefore, $I_F$ is surjective. From the previous discussion it follows:

\begin{thm}
    Let $F$ be a complete field satisfying the Stone-Weierstrass theorem. The category $\textsf{KH}_F$ is dual to the category of complete (w.r.t. the canonical uniformity) $F$-algebras satisfying \rom{1}-\rom{4} of Lemma~\ref{lem : propCAlg}.
    \label{thm : SWGD}
\end{thm}

Restating the previous result, we get the following characterization of algebras of continuous functions for these fields. 

\begin{cor}
Let $F$ be a complete field satisfying the Stone-Weierstrass theorem. An $F$-algebra $A$ is the algebra of $F$-valued continuous functions over a compact $F$-Tychonoff space iff $A$ is
\begin{enumerate}[label = \enum]
    \item  Gelfand.
    \item semisimple.
    \item the spectrum of every element is compact.
    \item every prime ideal is contained in a unique maximal ideal. 
    \item complete (w.r.t. the canonical uniformity).
\end{enumerate}
    \label{cor : ContAlgWS}
\end{cor}

There are plenty of topological fields satisfying the Stone-Weierstrass theorem. For instance, any topological field whose topology is induced by an absolute value or a Krull valuation satisfies this theorem (\cite{MR234939}). Unfortunately, there is no known characterization of such fields. 

In particular, both $\mathbb{R}$ and $\mathbb{C}$ satisfy the Stone-Weierstrass theorem. In the complex case, however, a modified version is required to account for complex conjugation. Thus, Theorem~\ref{thm : SWGD} gives us a purely algebraic proof of the Gelfand-Naimark theorem both real and complex. 

As we saw in Section~\ref{sec : 4}, if the field $F$ is disconnected and complete the Stone-Weierstrass theorem can be removed from the proof. Although this covers a vast family of fields, the arguments cannot be extended to other cases, as the result crucially relies on disconnectedness.

Moreover, if the field $F$ is complete, then completeness must be included as a necessary condition (Theorem~\ref{cor : Ccomp}). By the Remark after Theorem~\ref{thm: GelMap}, if a Gelfand algebra $A$ is complete, an assumption that only makes sense when the field itself is complete, then the image of $I_F$ is closed. Therefore, it suffices to show that it is also dense, which would complete the argument.

Ultimately, at least in the case of complete fields, the problem of establishing a duality reduces to determining necessary and sufficient conditions for the Gelfand map to have a dense image.

If the completeness assumption on the field is removed, then requiring algebras to be complete no longer makes sense. In fact, $C(X,F)$  is never complete when the field itself is incomplete. Consequently, the problem of characterizing the surjectivity of the Gelfand map becomes significantly more challenging. 

\bibliography{main.bib}
\bibliographystyle{abbrv}
\nocite{*}

\end{document}